\documentclass{amsart}
\usepackage{amssymb,amsthm,amstext,amsmath,amscd,latexsym,amsfonts}
\usepackage[all]{xy}

\def\CM{{\mathrm{CM}}}

\def\modR{\mathrm{mod}(R)}
\def\CMA{\mathrm{CM}(A)}
\def\CMR{\mathrm{CM}(R)}

\def\ModC{\mathrm{Mod}(\CMR)}
\def\modC{\mathrm{mod}(\CMR)}
\def\GR{\mathrm{G}(\CMR )}

\def\ulL{{\underline{L}}}
\def\ulM{{\underline{M}}}
\def\ulN{{\underline{N}}}
\def\ulZ{{\underline{Z}}}

\def\SCM{{\underline{\mathrm{CM}}}}
\def\SCMA{{\underline{\mathrm{CM}}}(A)}
\def\SCMR{{\underline{\mathrm{CM}}}(R)}

\def\SHom{{\underline{\mathrm{Hom}}}}
\def\SEnd{{\underline{\mathrm{End}}}}
\def\homo{\leq _{hom}}
\def\dego{\leq _{deg}}
\def\Dego{\leq _{DEG}}

\def\exto{\leq _{ext}}
\def\Exto{\leq _{EXT}}

\def\ARo{\leq _{AR}}

\def\trio{\leq _{tri}}
\def\sto{\leq _{st}}
\def\Hom{\mathrm{Hom}}

\def\N{\mathbb N}
\def\Z{\mathbb Z}
\def\m{\mathfrak m}
\def\p{\mathfrak p}
\def\q{\mathfrak q}
\def\L{\mathcal L} 
\def\R{\mathcal R} 
\def\Fitt{\mathcal F} 
\newtheorem{theorem}{Theorem}[section]
\newtheorem{lemma}[theorem]{Lemma}
\newtheorem{corollary}[theorem]{Corollary}
\newtheorem{proposition}[theorem]{Proposition}
\theoremstyle{definition}
\newtheorem{definition}[theorem]{Definition}
\newtheorem{example}[theorem]{Example}
\theoremstyle{remark}
\newtheorem{remark}[theorem]{Remark}
\numberwithin{equation}{section}
\begin{document}
%%%%%%%%%%%%%%%%%%%%%%%%%%%%%%%%%%%%%%%%%%%%%%%%%%%%%%%%
%%%%%%%%%%%%%%%%%%%%%%%%%%%%%%%%%%%%%%%%%%%%%%%%%%%%%%%%
%%%%%%%%%%%%%%%%%%%%%%%%%%%%%%%%%%%%%%%%%%%%%%%%%%%%%%%%
%%%%%%%%%%%%%%%%%%%%%%%%%%%%%%%%%%%%%%%%%%%%%%%%%%%%%%%%
%%%%%%%%%%%%% Title %%%%%%%%%%%%%%%%%%%%%%%%%%%%%%%%%%%%
%%%%%%%%%%%%%%%%%%%%%%%%%%%%%%%%%%%%%%%%%%%%%%%%%%%%%%%%
%%%%%%%%%%%%%%%%%%%%%%%%%%%%%%%%%%%%%%%%%%%%%%%%%%%%%%%%
%%%%%%%%%%%%%%%%%%%%%%%%%%%%%%%%%%%%%%%%%%%%%%%%%%%%%%%%
%%%%%%%%%%%%%%%%%%%%%%%%%%%%%%%%%%%%%%%%%%%%%%%%%%%%%%%%
\title[Examples of degenerations]{\textbf{Examples of degenerations of Cohen-Macaulay modules}}
%%%%%%%%%%%%%%%%%%%%%%%%%%%%%%%%%%%%%%%%%
%%%%%%%%%%%%%%%%%%%%%%%%%%%%%%%%%%%%%%%%%
%%%%%%%% Informations of Authors %%%%%%%%
%%%%%%%%%%%%%%%%%%%%%%%%%%%%%%%%%%%%%%%%%
%%%%%%%%%%%%%%%%%%%%%%%%%%%%%%%%%%%%%%%%%
\author{Naoya Hiramatsu} 
\address{Department of Mathematics, Graduate School of Natural Science and Technology, Okayama University, Okayama 700-8530, Japan}
\email{naoya$\_$h@math.okayama-u.ac.jp, yoshino@math.okayama-u.ac.jp}

\author{Yuji Yoshino}
%\address{Department of Mathematics, Okayama University, 700-8530, Okayama, Japan}
%%%%%%%%%%%%%%%%%%%%%%%%%%%%%%%%%%%%%%%%%
%%%%%%%%%%%%%%%%%%%%%%%%%%%%%%%%%%%%%%%%%
%%%%%%%%%%%% General info %%%%%%%%%%%%%%%
%%%%%%%%%%%%%%%%%%%%%%%%%%%%%%%%%%%%%%%%%
%%%%%%%%%%%%%%%%%%%%%%%%%%%%%%%%%%%%%%%%%
\date{\today}
%%%%%%%%%%%%%%%%%%%%%%%%%%%%%%%%%%%%%%%%%
%%%%%%%%%%%%%%%%%%%%%%%%%%%%%%%%%%%%%%%%%
%%%%%%%%%%%%%%%%%%%%%%%%%%%%%%%%%%%%%%%%%
%%%%%%%%%%%%% Abstract %%%%%%%%%%%%%%%%%%
%%%%%%%%%%%%%%%%%%%%%%%%%%%%%%%%%%%%%%%%%
%%%%%%%%%%%%%%%%%%%%%%%%%%%%%%%%%%%%%%%%%
%%%%%%%%%%%%%%%%%%%%%%%%%%%%%%%%%%%%%%%%%
\begin{abstract}
We study the degeneration problem for maximal Cohen-Macaulay modules and give several examples of such degenerations. 
It is proved that such degenerations over an even-dimensional simple hypersurface singularity of type $(A_n)$ are given by extensions. 
We also prove that all extended degenerations of maximal Cohen-Macaulay modules over a Cohen-Macaulay complete local algebra of finite representation type are obtained by iteration of extended degenerations of Auslander-Reiten sequences. 
\end{abstract}
\maketitle
%%%%%%%%%%%%%%%%%%%%%%%%%%%%%%%%%%%%%%%%%%
%%%%%%%%%%%%%%%%%%%%%%%%%%%%%%%%%%%%%%%%%%
%%%%%%%%%%%%%%%%%%%%%%%%%%%%%%%%%%%%%%%%%%
%%%%%%% Introduction %%%%%%%%%%%%%%%%%%%%%
%%%%%%%%%%%%%%%%%%%%%%%%%%%%%%%%%%%%%%%%%%
%%%%%%%%%%%%%%%%%%%%%%%%%%%%%%%%%%%%%%%%%%
%%%%%%%%%%%%%%%%%%%%%%%%%%%%%%%%%%%%%%%%%%
%%%%%%%%%%%%%%%%%%%%%%%%%%%%%%%%%%%%%%%%%%
\section{Introduction} 
The degeneration problem of modules has been studied by many authors \cite{B96, R86, Y02, Y04, Z99}. 
For modules over an artinian algebra, it has been studied by Bongartz \cite{B96} in relation with the Auslander-Reiten quiver. 
In general but for modules over a commutative noetherian ring, the second author \cite{Y02} generalized the theory of Bongartz to maximal Cohen-Macaulay modules and he has shown that any extended degenerations of maximal Cohen-Macaulay modules are fundamentally obtained by degenerations of Auslander-Reiten sequences under certain special conditions. 
For this, several order relations for modules, such as the hom order, the degeneration order, the extension order and the AR order, were introduced, 
and the connection among them has been studied. 

The purpose of this paper is to give several examples of degenerations of maximal Cohen-Macaulay modules and to show how we can describe them. 
We will be able to give the complete description of degenerations over a ring of even-dimensional simple hypersurface singularity of type ($A_n$). 
In fact, we will show that all degenerations of maximal Cohen-Macaulay modules over such a ring are given by extensions (Theorem \ref{type A}). 
This result depends heavily on the recent work of the second author about 
 the stable analogue of degenerations for Cohen-Macaulay modules over a Gorenstein local algebra \cite{Y10}.

In Section \ref{Orders} we also investigate the relation among the extended versions of the degeneration order, the extension order and the AR order. 
As a result  we shall show that if $R$ is of finite Cohen-Macaulay representation type, then all these extended orders are identical when restricted on the maximal Cohen-Macaulay modules (Theorem \ref{theorem of AR order}). 

%%%%%%%%%%%%%%%%%%%%%%%%%%%%%%%%%%%%%%%%%%%%%%%%%%%%%%%%%
%%%%%%%%%%%%%%%%%%%%%%%%%%%%%%%%%%%%%%%%%%%%%%%%%%%%%%%%%
%%%%%%%%%%%%%%%%%%%%%%%%%%%%%%%%%%%%%%%%%%%%%%%%%%%%%%%%%
%%%%%%%%%%%%%%%%%%%%%%% Section 1 %%%%%%%%%%%%%%%%%%%%%%%
%%%%%%%%%%%%%%%%%%%%%%%%%%%%%%%%%%%%%%%%%%%%%%%%%%%%%%%%%
%%%%%%%%%%%%%%%%%%%%%%%%%%%%%%%%%%%%%%%%%%%%%%%%%%%%%%%%%
%%%%%%%%%%%%%%%%%%%%%%%%%%%%%%%%%%%%%%%%%%%%%%%%%%%%%%%%%
%%%%%%%%%%%%%%%%%%%%%%%%%%%%%%%%%%%%%%%%%%%%%%%%%%%%%%%%%
\section{Preliminaries and the first examples}\label{First example}

In this section, we recall the definition of degeneration and state several known results on degenerations. 
For the detail, we recommend the reader to refer to \cite{Y04, Y10}.

%%%%%%%%%%%%%%%%%%%%%%%%%%%%%%%%%%%%%%%%%%%%%
%%%%%%%%%%%%%%%%%%%%%%%%%%%%%%%%%%%%%%%%%%%%%
%%%%%%%%%% Definition of degenertions %%%%%%%
%%%%%%%%%%%%% of f. g. modules %%%%%%%%%%%%%%
%%%%%%%%%%%%%%%%%%%%%%%%%%%%%%%%%%%%%%%%%%%%%
%%%%%%%%%%%%%%%%%%%%%%%%%%%%%%%%%%%%%%%%%%%%%
\begin{definition}\label{degeneration}
Let $R$ be a noetherian algebra over a field $k$, and let $M$ and $N$ be finitely generated left $R$-modules. 
We say that $M$ degenerates to $N$, or $N$ is a degeneration of $M$, if there is a discrete valuation ring $(V, tV, k)$ that is a $k$-algebra (where $t$ is a prime element) and a finitely generated left $R\otimes _{k} V$-module $Q$ which satisfies the following conditions: 

\begin{itemize}
\item[(1)] $Q$ is flat as a $V$-module.

\item[(2)] $Q/tQ \cong N$ as a left $R$-module. 

\item[(3)] $Q[1/t] \cong M\otimes _{k} V[1/t]$ as a left $R\otimes _{k} V[1/t]$-module.
\end{itemize}
\end{definition}

The following characterization of degenerations has been proved by the second author \cite{Y04}. 
See also \cite{R86, Z00}. 
%%%%%%%%%%%%%%%%%%%%%%%%%%%%%%%%%%%%%%%%%%%%%
%%%%%%%%%%%%%%%%%%%%%%%%%%%%%%%%%%%%%%%%%%%%%
%%%%%%%%%%%%%%%%%%%%%%%%%%%%%%%%%%%%%%%%%%%%%
%%%%%%%% Theorem of Zwara's sequence %%%%%%%%
%%%%%%%%%%%%%%%%%%%%%%%%%%%%%%%%%%%%%%%%%%%%%
%%%%%%%%%%%%%%%%%%%%%%%%%%%%%%%%%%%%%%%%%%%%%
\begin{theorem}\label{Zwara sequence}\cite[Theorem 2.2]{Y04}
The following conditions are equivalent for finitely generated left $R$-modules $M$ and $N$.

\begin{itemize}
\item[(1)] $M$ degenerates to $N$. 

\item[(2)] There is a short exact sequence of finitely generated left $R$-modules
$$
\begin{CD}
0 @>>> Z @>{\tiny 
\begin{pmatrix}
\varphi \\
\psi  \\
\end{pmatrix} 
}>> M\oplus Z \ @>>> N @>>> 0,   \\ 
\end{CD}
$$ 
such that the endomorphism $\psi$ of $Z$ is nilpotent, {\it i.e.} $\psi ^n = 0$ for $n\gg 1$.
\end{itemize}
\end{theorem}

%%%%%%%%%%%%%%%%%%%%%%%%%%%%%%%%%%%%%%%%%%%%%
%%%%%%%%%%%%%%%%%%%%%%%%%%%%%%%%%%%%%%%%%%%%%
%%%%%%%%%%%%%%%%%%%%%%%%%%%%%%%%%%%%%%%%%%%%%
%%%%%%% Remark of Extaensions %%%%%%%%%%%%%%%
%%%%%%%%%%%%%%%%%%%%%%%%%%%%%%%%%%%%%%%%%%%%%
%%%%%%%%%%%%%%%%%%%%%%%%%%%%%%%%%%%%%%%%%%%%%
%%%%%%%%%%%%%%%%%%%%%%%%%%%%%%%%%%%%%%%%%%%%%
\begin{remark}\label{remark of extension}
Assume that there is an exact sequence of finitely generated left $R$-modules 
$$
\begin{CD}
0 @>>> L @>>> M @>>> N @>>> 0. 
\end{CD}
$$ 
Then $M$ degenerates to $L \oplus N$. 
See \cite[Remark 2.5]{Y04}. 
\end{remark}

%%%%%%%%%%%%%%%%%%%%%%%%%%%%%%%%%%%%%%%%%%%%%
%%%%%%%%%%%%%%%%%%%%%%%%%%%%%%%%%%%%%%%%%%%%%
%%%%%%%%%%%%%%%%%%%%%%%%%%%%%%%%%%%%%%%%%%%%%
We are mainly interested in degenerations of modules over commutative rings. 
Henceforth, in the rest of the paper, all the rings are assumed to be commutative.
%%%%%%%%%%%%%%%%%%%%%%%%%%%%%%%%%%%%%%%%%%%%%
%%%%%%%%%%%%%%%%%%%%%%%%%%%%%%%%%%%%%%%%%%%%%
%%%%%%%%%%%%%%%%%%%%%%%%%%%%%%%%%%%%%%%%%%%%%
%%%% Definitions of deg and ext order %%%%%%%
%%%%%%%%%%%%%%%%%%%%%%%%%%%%%%%%%%%%%%%%%%%%%
%%%%%%%%%%%%%%%%%%%%%%%%%%%%%%%%%%%%%%%%%%%%%
%%%%%%%%%%%%%%%%%%%%%%%%%%%%%%%%%%%%%%%%%%%%%
\begin{definition}\label{ext order}
Let  $M$  and $N$  be finitely generated modules over a commutative noetherian $k$-algebra $R$.  

\begin{itemize}
\item[(1)]
We denote by $M \dego N$ if $N$ is obtained from $M$ by iterative degenerations, i.e.  there is a sequence of finitely generated $R$-modules $L_0, L_1, \ldots , L_r$  such that  $M \cong L_0$, $N \cong L_r$  and  each $L_{i}$  degenerates to  $L_{i+1}$  for $0 \leq i < r$. 

\item[(2)]
We say that $M$ degenerates by an extension to $N$ 
if there is a short exact sequence $0 \to U \to M \to V \to 0$ of finitely generated $R$-modules such that $N \cong U \oplus N$. 

We denote by $M \exto N$ if $N$ is obtained from $M$ by iterative degenerations  by extensions, i.e.  there is a sequence of finitely generated $R$-modules $L_0, L_1, \ldots , L_r$  such that  $M \cong L_0$, $N \cong L_r$  and  each $L_{i}$  degenerates by an extension to  $L_{i+1}$  for $0 \leq i < r$. 
\end{itemize}
\end{definition}

If $R$  is a local ring, then  $\dego$ and $\exto$ are known to be partial orders on the set of isomorphism classes of finitely generated R-modules, which are called  the degeneration order and the extension order respectively. 
See \cite{Y02} for the detail.

\begin{remark}
By virtue of Remark \ref{remark of extension}, 
if  $M \exto N$ then  $M \dego N$. 
However the converse is not necessarily true. 

For example, consider a ring $R = k[[x, y]]/(x^2)$. 
A pair of matrices over $k[[x,y]]$; 
$$
(\varphi, \psi) = \left( \begin{pmatrix} x & y^2 \\ 0 & x \end{pmatrix}, 
\begin{pmatrix} x & -y^2 \\ 0 & x \end{pmatrix} \right)
$$
is a matrix factorization of the equation $x^2$, hence it gives a maximal Cohen-Macaulay $R$-module $N$ that is isomorphic to the ideal $(x, y^2)R$. 
Actually there is an exact sequence 
$$
\begin{CD}
 \cdots @>>>R^2  @>{\psi}>> R^2  @>{\varphi}>>R^2  @>{\psi}>> R^2  @>{\varphi}>> R^2 @>>> N @>>> 0. 
\end{CD}
$$
It is known that  $N$ is indecomposable. 
See \cite[Example (6.5)]{Y}.  
Now we deform the matrices $(\varphi, \psi)$ to 
$$
(\Phi, \Psi) = \left( \begin{pmatrix} x + ty & y^2 \\ -t^2 & x -ty \end{pmatrix}, \begin{pmatrix} x-ty  & -y^2 \\ t^2  & x+ty  \end{pmatrix} \right)
$$
over $R \otimes _kV$, where $V = k[t]_{(t)}$. 
Then, since  $(\Phi, \Psi)$  is still a matrix factorization of $x^2$ over the regular ring $k[[x, y]] \otimes _k V$, 
we have an exact sequence 
$$
\begin{CD}
\cdots @>{\Phi}>> (R\otimes _kV)^2 @>{\Psi}>> (R\otimes _kV)^2  @>{\Phi}>> (R\otimes _kV)^2 @>>> Q @>>> 0, 
\end{CD}
$$
where  $Q$  is a maximal Cohen-Macaulay module over $R \otimes _k V$, in particular  $Q$  is $V$-flat. 
Since $\Phi \otimes _V V/tV = \varphi$, it follows that $Q/tQ \cong N$.  
On the other hand, since $t^2$ is a unit in $R\otimes_k V[1/t]$, 
after an elementary transformation of matrices, we have $\Phi \otimes_V V[1/t] \cong \begin{pmatrix} 0 & 0 \\ 1 & 0 \end{pmatrix}$, hence $Q_t \cong R \otimes _k V[1/t]$. 
As a result we see that $R$ degenerates to $(x, y^2)R$ in this case, and hence  $R \dego (x, y^2)R$. 

In general if  $M \exto N$ and if $M \not\cong N$,  then $N$ is a non-trivial direct sum of modules.  
Since  $N \cong (x, y^2)R$  is indecomposable, we see that  $R \exto (x, y^2)R$  can never happen. 
\end{remark}

%%%%%%%%%%%%%%%%%%%%%%%%%%%%%%%%%%%%%%%%%%%%%
%%%%%%%%%%%%%%%%%%%%%%%%%%%%%%%%%%%%%%%%%%%%%
%%%%%%%%%%%%%%%%%%%%%%%%%%%%%%%%%%%%%%%%%%%%%
%%%%%%% transitive condition %%%%%%%%%%%%%%%%
%%%%%%%%%%%%%%%%%%%%%%%%%%%%%%%%%%%%%%%%%%%%%
%%%%%%%%%%%%%%%%%%%%%%%%%%%%%%%%%%%%%%%%%%%%%
%%%%%%%%%%%%%%%%%%%%%%%%%%%%%%%%%%%%%%%%%%%%%
We also note that if $R$ is an artinian $k$-algebra, then the degeneration for finitely generated modules is known to be transitive by Zwara \cite{Z98}. 
Namely, if $L$ degenerates to $M$ and if $M$ degenerates to $N$, then $L$ degenerates to $N$. 
In general if $R$  is not necessarily artinian, we do not know whether this transitivity property holds or not. 
However, it is rather easy to see the following:

\begin{remark}
If $L$  degenerates to $M$ and if $M$ degenerates by an extension to $N$, then $L$ degenerates to $N$. 
\end{remark}

%%%%%%%%%%%%%%%%%%%%%%%%%%%%%%%%%%%%%%%%%%%%%
%%%%%%%%%%%%%%%%%%%%%%%%%%%%%%%%%%%%%%%%%%%%%
%%%%%%%%%%%% R/I %%%%%%%%%%%%%%%%%%%%%%%%%%%%
%%%%%%%%%%%%%%%%%%%%%%%%%%%%%%%%%%%%%%%%%%%%%
%%%%%%%%%%%%%%%%%%%%%%%%%%%%%%%%%%%%%%%%%%%%%
Now we note that the following lemma holds.  

\begin{lemma}\label{lemma R/I} 
Let $I$  be an ideal of a noetherian $k$-algebra $R$, 
and let $M$ and $N$ be finitely generated $R/I$-modules. 
Then  $M$ degenerates (resp. degenerates by an extension) to $N$ as $R$-modules  if and only if so does as $R/I$-modules. 
\end{lemma}

\begin{proof} 
Assume $M$ degenerates to $N$ as $R/I$-modules. 
Then there is a finitely generated $R/I \otimes _k V$-module $Q$ satisfying the conditions in Definition \ref{degeneration}. 
Regarding $Q$ as an $R \otimes _k V$-module, we can see that 
$M$ degenerates to $N$ as $R$-modules. 

Contrarily assume $M$ degenerates to $N$ as $R$-modules. 
Then the $R \otimes _k V$-module $Q$ satisfying the conditions in Definition \ref{degeneration}  is an $R/I \otimes _k V$-module. 
In fact, since  $IM = 0$ and since  $Q [1/t] \cong M \otimes _k V[1/t]$, we have $IQ[1/t]=0$. 
On the other hand, since  $Q$ is V-flat, the natural mapping  $Q \to Q[1/t]$  is injective. 
Therefore we see that  $IQ =0$, which shows that  $Q$ is a module over $R/I\otimes _k V$. 
Then $Q$ satisfies all the conditions in Definition \ref{degeneration} as $R/I$-modules. 
Hence  $M$ degenerates to $N$ as $R/I$-modules. 

Assume $M$ degenerates by an extension to $N$ as $R$-modules. 
Then there is an exact sequence $0 \to U \to M \to V \to 0$ with  $N \cong U \oplus V$. 
Since  $IM =0$, it follows that $U$ and $V$ are also $R/I$-modules and the short exact sequence is so as $R/I$-modules. 
Hence  $M$ degenerates by an extension to $N$ as $R/I$-modules.  
\end{proof}

By iterative use of this lemma we have the following corollary. 

\begin{corollary}\label{cor R/I}
As in the lemma,  let $I$  be an ideal of a noetherian $k$-algebra $R$, 
and let $M$ and $N$ be finitely generated $R/I$-modules. 
Then  $M \dego N$ (resp. $M \exto N$) as $R$-modules if and only if so does as $R/I$-modules. 
\end{corollary}

%%%%%%%%%%%%%%%%%%%%%%%%%%%%%%%%%%%%%%%%%%%%%
%%%%%%%%%%%%%%%%%%%%%%%%%%%%%%%%%%%%%%%%%%%%%
%%%%%%%%%%%%%%%%%%%%%%%%%%%%%%%%%%%%%%%%%%%%%
%%%%% Remark of degenerations %%%%%%%%%%%%%%%
%%%%%%%%%%%%%%%%%%%%%%%%%%%%%%%%%%%%%%%%%%%%%
%%%%%%%%%%%%%%%%%%%%%%%%%%%%%%%%%%%%%%%%%%%%%
%%%%%%%%%%%%%%%%%%%%%%%%%%%%%%%%%%%%%%%%%%%%%

We make several other remarks on degenerations for the later use. 

\begin{remark}\label{remark of degenerations}
Let $R$ be a noetherian $k$-algebra, and let $M$ and $N$ be finitely generated $R$-modules. 
Suppose that $M$ degenerates to $N$. 
Then the following hold. 
\begin{itemize}
\item[(1)] The modules $M$ and $N$ give the same class in the Grothendieck group, {\it i.e.} $[M] = [N]$ as an element of $K_{0} (\modR)$, where $\modR$ denotes the category of finitely generated $R$-modules and $R$-homomorphisms. 
(See \cite[Remark 2.3 (1)]{Y10}). 

\item[(2)]  The $i$th Fitting ideal of $M$ contains that of $N$ for all $i \geq 0$. 
Namely, denoting the $i$th Fitting ideal of an $R$-module $M$ by $\Fitt_{i}^{R} (M)$, we have $\Fitt_{i}^{R} (M) \supseteq \Fitt_{i}^{R} (N)$ for all $i \geqq 0$. 
(See \cite[Theorem 2.5]{Y10}). 
\end{itemize}
\end{remark}

%%%%%%%%%%%%%%%%%%%%%%%%%%%%%%%%%%%%%%%%%%%%%
%%%%%%%%%%%%%%%%%%%%%%%%%%%%%%%%%%%%%%%%%%%%%
%%%%%%%%%%%%%%%%%%%%%%%%%%%%%%%%%%%%%%%%%%%%%
%%%%%%%%%%%%%% Krull-Schmidt %%%%%%%%%%%%%%%%
%%%%%%%%%%%%%%%%%%%%%%%%%%%%%%%%%%%%%%%%%%%%%
%%%%%%%%%%%%%%%%%%%%%%%%%%%%%%%%%%%%%%%%%%%%%
%%%%%%%%%%%%%%%%%%%%%%%%%%%%%%%%%%%%%%%%%%%%%
If $R$ is a complete (more generally, Henselian) local ring, then it is known that the category $\modR$ of finitely generated $R$-modules and $R$-homomorphisms  is a Krull-Schmidt category, i.e. every finitely generated $R$-module is uniquely a finite direct sum of indecomposable $R$-modules. 
To use this property, all the rings considered below are assumed to be complete local rings.

\vspace{6pt}
%%%%%%%%%%%%%%%%%%%%%%%%%%%%%%%%%%%%%%%%%%%%%
%%%%%%%%%%%%%%%%%%%%%%%%%%%%%%%%%%%%%%%%%%%%%
%%%%%%%%%%%%%%%%%%%%%%%%%%%%%%%%%%%%%%%%%%%%%
%%%%%%%%%%%%%%%%%%%%%%%%%%%%%%%%%%%%%%%%%%%%%
Now we give an example of modules of finite length for which we can easily describe the degeneration.

Let $R = k[[ x ]]$ be a formal power series ring over a field $k$ with one variable $x$ and let $M$ be an $R$-module of length $n$. 
It is easy to see that there is an isomorphism 
\begin{equation}\label{partition1}
M \cong R/(x^{p_1}) \oplus \cdots \oplus R/(x^{p_n}),
\end{equation}  
where 
\begin{equation}\label{partition2}
p_1 \geq p_2 \geq \cdots \geq p_n \geq 0  \quad \text{and} \quad \sum _{i = 1}^{n}p_i = n. 
\end{equation}  
In this case the finite presentation of $M$ is given as follows:
$$
\begin{CD}
0 @>>> R^{n} @>{\left( 
\begin{array}{ccc}
x^{p_1} & &  \\
 & \ddots &  \\
 &  & x^{p_n} \\
\end{array}
\right)
}>> R^n \ @>>> M @>>> 0. \\
\end{CD}
$$
Note that we can easily compute the $i$th Fitting ideal of $M$ from this presentation; 
$$
\Fitt_{i}^{R} (M) = (x^{p_{i+1} + \cdots + p_n}) \ \ (i \geq 0).
$$
We denote by $p_M$ the sequence $(p_1, p_2, \cdots , p_n)$ of non-negative integers. 
Recall that such a sequence satisfying (\ref{partition2}) is called a partition of $n$. 

Conversely, given a partition $p = (p_1, p_2, \cdots , p_n)$ of $n$,  we can associate  an $R$-module of length $n$  by (\ref{partition1}), which we denote by  $M(p)$. 
In such a way we see that there is a one-one correspondence between the set of partitions of $n$  and the set of isomorphism classes of  $R$-modules of length $n$.   

%%%%%%%%%%%%%%%%%%%%%%%%%%%%%%%%%%%%%%%%%%%%%
%%%%%%%%%%%%%%%%%%%%%%%%%%%%%%%%%%%%%%%%%%%%%
%%%%%%%%%%%%%%%%%%%%%%%%%%%%%%%%%%%%%%%%%%%%%
%%%%% Definition of dominance order %%%%%%%%%
%%%%%%%%%%%%%%%%%%%%%%%%%%%%%%%%%%%%%%%%%%%%%
%%%%%%%%%%%%%%%%%%%%%%%%%%%%%%%%%%%%%%%%%%%%%
%%%%%%%%%%%%%%%%%%%%%%%%%%%%%%%%%%%%%%%%%%%%%
\begin{definition}\label{dominance order}
Let $n$ be a positive integer and let  $p = (p_1,p_2, \cdots , p_n)$ and $q = (q_1,q_2, \cdots , q_n)$ be partitions of $n$. 
Then we denote  $p \succeq q$  if it satisfies $\sum _{i = 1}^{j}p_i \geq  \sum _{i = 1}^{j}q_i$ for all $1 \leq j \leq n$. 
\end{definition}

We note that $\succeq $ is known to be a partial order on the set of partitions of $n$  and called the dominance order. 
See for example \cite[page 7]{Mac}.

\vspace{6pt}
It is known that the degeneration problem for $R$-modules of length $n$ is equivalent to the degeneration problem for Jordan canonical forms of square matrices of size $n$. 
In the following proposition we show the degeneration order for $R$-modules of length $n$  coincides with the opposite of the dominance order of corresponding partitions. 
Note that if  $M$ and $N$ are $R$-modules of finite length and if $M$ degenerates to $N$, then the length of $M$ equals the length of $N$, since  $[M] = [N]$ in the Grothendieck group.   

%%%%%%%%%%%%%%%%%%%%%%%%%%%%%%%%%%%%%%%%%%%%%
%%%%%%%%%%%%%%%%%%%%%%%%%%%%%%%%%%%%%%%%%%%%%
%%%%%%%%%%%%%%%%%%%%%%%%%%%%%%%%%%%%%%%%%%%%%
%%%% Degenerations problem of %%%%%%%%%%%%%%%
%%%%%%%%%% Jordan canonical forms %%%%%%%%%%%
%%%%%%%%%%%%%%%%%%%%%%%%%%%%%%%%%%%%%%%%%%%%%
%%%%%%%%%%%%%%%%%%%%%%%%%%%%%%%%%%%%%%%%%%%%%

\begin{proposition}\label{Jordan}
Let $R = k[[ x ]]$ as above, and let $M$, $N$ be $R$-modules of length $n$. 
Then the following conditions are equivalent: 
\begin{itemize}
\item[$(1)$]
$M \dego N$, 
\item[$(2)$]
$M \exto N$, 
\item[$(3)$]
$p_M \succeq p_N$.
\end{itemize}
\end{proposition}

%%%%%%%%%%%%%%%%%%%%%%%%%%%%%%%%%%%%%%%%%%%%%
%%%%%%%%%%%%%%%%%%%%%%%%%%%%%%%%%%%%%%%%%%%%%
%%%%% Proof for Prop. %%%%%%%%%%%%%%%%%%%%%%%
%%%%%%%%%%%%%%%%%%%%%%%%%%%%%%%%%%%%%%%%%%%%%
%%%%%%%%%%%%%%%%%%%%%%%%%%%%%%%%%%%%%%%%%%%%%
\begin{proof}
First of all, we assume $M$ degenerates to $N$, and let $p_M = (p_1, p_2, \cdots ,p_n)$ and $p_N = (q_1, q_2, \cdots ,q_n)$.  
Then, by definition, we have the equalities of the Fitting ideals;  $\Fitt_{i}^{R} (M) = (x^{p_{i+1} + \cdots + p_n})$ and $\Fitt_{i}^{R} (M) = (x^{q_{i+1} + \cdots + q_n})$ for all $i \geq 0$. 
Since $M$ degenerates to $N$, it follows from Remark \ref{remark of degenerations}(2) that  $\Fitt_{i}^{R} (M) \supseteq \Fitt_{i}^{R} (N)$ for all $i$. 
Thus $p_{i+1} + \cdots + p_n \leq q_{i+1} + \cdots + q_{n}$. 
Since $\sum _{i = 1}^{n}p_i = n = \sum _{i = 1}^{n}q_i$, it follows that $p_1 + \cdots + p_i \geq  q_1 + \cdots + q_i$ for all $i \geq 0$. 
Therefore  $p_M \succeq p_N$. 

Secondly, assume $M \dego N$. 
Then there are $R$-modules  $L_0, L_1, \ldots , L_r$  such that  $M \cong L_0$, $N \cong L_r$  and  each $L_{i}$  degenerates to  $L_{i+1}$  for $0 \leq i < r$. 
It then follows from the above that 
$p_{L_0} \succeq p_{L_1} \succeq \cdots  \succeq p_{L_r}$. 
Since $\succeq$  is a partial order, we have  $p_{M} \succeq p_{N}$. 
Thus we have proved the implication $(1) \Rightarrow (3)$. 

Finally we shall prove $(3) \Rightarrow (2)$. 
To this end let $p = (p_1,p_2, \cdots , p_n)$ and $q = (q_1,q_2, \cdots , q_{n})$ be partitions of $n$. 
Note that it is enough to prove that the corresponding $R$-module  $M(p)$  degenerates by an extension to $N(q)$  whenever  $q$ is a predecessor of $p$ under the dominance order. 
(Recall that $q$ is called a predecessor of $p$ if $p \succeq q$  and there are no partitions $r$ with  $p \succeq r \succeq q$ other than  $p$ and $q$.) 

Assume that  $q$ is a predecessor of $p$ under the dominance order. 
Then it is easy to see that  there are numbers  $1 \leq i < j \leq  n$  with  $p_i - p_j \geq 2$, $p_i > p_{i+1}$, $p_{j-1} > p_{j}$ such that 
the equality $q = (p_1, \cdots ,p_{i}-1, p_{i+1},\cdots ,p_{j} + 1, \cdots , p_n)$ holds. 
In this case, setting  $L = M((p_1, \cdots ,p_{i-1}, p_{i+1},\cdots ,p_{j-1}, p_j,  \cdots , p_n))$,  we have 
$M(p) = L \oplus M((p_i, p_j))$  and  $M(q) = L \oplus M((p_i-1, p_j+1))$. 
Note that, in general, if $M$ degenerates by an extension to $N$, then $M \oplus L$ degenerates by an extension to $N \oplus L$,  for any $R$-modules $L$. 
Hence it is enough to show that  $M((a, b))$   degenerates by an extension  to   $M((a-1, b+1))$ if  $a \geq b+2$. 
However there is a short exact sequence of the form: 
$$
\begin{CD}
0 @>>> R/(x^{a-1}) @>>> R/(x^{a}) \oplus R/(x^{b}) \ @>>> R/(x^{b+1}) @>>> 0 \\ @. 1  @>>> (x, 1) @. @. 
\end{CD}
$$
Thus  $M((a, b))= R/(x^{a}) \oplus R/(x^{b})$ degenerates by an extension to $M((a-1, b+1)) = R/(x^{a-1}) \oplus R/(x^{b+1})$. 
\end{proof}

Combining Proposition \ref{Jordan} with Corollary \ref{cor R/I}, 
we have the following corollary which will be used in the next section. 

\begin{corollary}\label{cor Jordan}
Let $R = k[[ x ]]/(x^m)$,  where $k$ is a field and $m$ is a positive integer, and  let $M$, $N$ be finitely generated $R$-modules. 
Then  $M \dego N$  holds if and only if $M \exto N$ holds. 
\end{corollary}

\vspace{6pt}
%%%%%%%%%%%%%%%%%%%%%%%%%%%%%%%%%%%%%%%%%%%%%%%%%%%%%%%%%
%%%%%%%%%%%%%%%%%%%%%%%%%%%%%%%%%%%%%%%%%%%%%%%%%%%%%%%%%
%%%%%%%%%%%%%%%%%%%%%%%%%%%%%%%%%%%%%%%%%%%%%%%%%%%%%%%%%
%%%%%%%%%%%%%%%%%%%%%%% Section 2 %%%%%%%%%%%%%%%%%%%%%%%
%%%%%%%%%%%%%%%%%%%%%%%%%%%%%%%%%%%%%%%%%%%%%%%%%%%%%%%%%
%%%%%%%%%%%%%%%%%%%%%%%%%%%%%%%%%%%%%%%%%%%%%%%%%%%%%%%%%
%%%%%%%%%%%%%%%%%%%%%%%%%%%%%%%%%%%%%%%%%%%%%%%%%%%%%%%%%
%%%%%%%%%%%%%%%%%%%%%%%%%%%%%%%%%%%%%%%%%%%%%%%%%%%%%%%%%
\section{The Second examples}\label{Second example}
Let  $k$  be a field and 
 $R = k[[ x_0, x_1, x_2, \cdots , x_d ]] / (f)$, 
 where $f$ is a polynomial of the form 
$$
f = x_0 ^{n + 1} + x_1 ^2 +x_2 ^2 + \cdots + x_d ^2  \qquad (n \geq 1).
$$
Recall that such a ring $R$ is call the ring of simple singularity of type $(A_n)$. 
Note that $R$ is a Gorenstein complete local ring and has finite Cohen-Macaulay representation type (cf. \cite{Y}). 
In this section, we focus on the degeneration problem of maximal Cohen-Macaulay modules over the ring $R$ of simple singularity of type $(A_n)$ and of even dimension. 
The main result of this section is the following whose proof will be given in the last part of this section.

%%%%%%%%%%%%%%%%%%%%%%%%%%%%%%%%%%%%%%%%%%%%%
%%%%%%%%%%%%%%%%%%%%%%%%%%%%%%%%%%%%%%%%%%%%%
%%%%%%%%%%%%%%%%%%%%%%%%%%%%%%%%%%%%%%%%%%%%%
%%%%% Degeneration problem of %%%%%%%%%%%%%%%
%%%%%%%%% (A_n) representation type %%%%%%%%%
%%%%%%%%%%%%%%%%%%%%%%%%%%%%%%%%%%%%%%%%%%%%%
%%%%%%%%%%%%%%%%%%%%%%%%%%%%%%%%%%%%%%%%%%%%%
\begin{theorem}\label{type A}
Let  $k$ be an algebraically closed field of characteristic $0$ and 
let $R = k[[ x_0, x_1, x_2, \cdots , x_d ]] / (x_0 ^{n + 1} + x_1 ^2 +x_2 ^2 + \cdots + x_d ^2)$ as above, 
where we assume that $d$ is even. 
For maximal Cohen-Macaulay $R$-modules $M$ and $N$, if $M \dego N$, then  $M \exto N$.
\end{theorem}

%%%%%%%%%%%%%%%%%%%%%%%%%%%%%%%%%%%%%%%%%%%%%
%%%%%%%%%%%%%%%%%%%%%%%%%%%%%%%%%%%%%%%%%%%%%
%%%%%%%%%%%%%%%%%%%%%%%%%%%%%%%%%%%%%%%%%%%%%
To prove the theorem, we need several results concerning the stable degeneration which was introduced by the second author in \cite{Y10}.

\vspace{6pt}
Let $A$ be a commutative Gorenstein ring. 
We denote by $\CMA$ the category of all maximal Cohen-Macaulay $A$-module with all $A$-homomorphisms. 
And we also denote by $\SCMA$ the stable category of $\CMA$. 
Recall that the objects of $\SCMA$ are the same as those of $\CMA$, and the morphisms of $\SCMA$ are elements of $\SHom _A(M, N) = \Hom _A(M, N)/P(M, N)$ for $M, N \in \SCMA$, where $P(M, N)$ denote the set of morphisms from $M$ to $N$ factoring through projective $A$-modules.
For a maximal Cohen-Macaulay module $M$ we denote it by $\ulM$ to indicate that it is an object of $\SCMA$. 
Since $A$ is Gorenstein, it is known that $\SCMA$ has a structure of triangulated category.
By definition, $ \ulL \to \ulM \to \ulN \to \ulL [1]$ is a triangle in $\SCMA$  if and only if there is an exact $0 \to L' \to M' \to N' \to 0$  in  $\CMA$ with $\ulL ' \cong \ulL$, $\ulM ' \cong \ulM$  and $\ulN ' \cong \ulN$  in $\SCMA$. 
See \cite[Chapter 1]{H}, \cite[Section 4]{Y10} for the detail.

%%%%%%%%%%%%%%%%%%%%%%%%%%%%%%%%%%%%%%%%%%%%%
%%%%%%%%%%%%%%%%%%%%%%%%%%%%%%%%%%%%%%%%%%%%%
%%%%%%%%%%%%%%%%%%%%%%%%%%%%%%%%%%%%%%%%%%%%%
\vspace{6pt}
Let $(R, \m, k)$ be a Gorenstein local ring that is $k$-algebra and let $V = k[t]_{(t)}$ and $K = k(t)$. 
Note that $R\otimes _k V$ and $R\otimes _k K$ are Gorenstein rings as well. 
Hence, as mentioned above, $\SCM (R\otimes _k V)$ and $\SCM (R\otimes _k K)$ are triangulated categories. 
We denote by $\L : \SCM (R\otimes _k V) \to \SCM (R\otimes _k K)$ (resp. $\R : \SCM (R\otimes _k V) \to \SCMR$) the triangle functor defined by the localization by $t$ (resp. taking $-\otimes _{V} V/tV$).

%%%%%%%%%%%%%%%%%%%%%%%%%%%%%%%%%%%%%%%%%%%%%
%%%%%%%%%%%%%%%%%%%%%%%%%%%%%%%%%%%%%%%%%%%%%
%%%%%%%%%%%%%%%%%%%%%%%%%%%%%%%%%%%%%%%%%%%%%
%%%% Definition of stably degeneration %%%%%%
%%%%%%%%%%%%%%%%%%%%%%%%%%%%%%%%%%%%%%%%%%%%%
%%%%%%%%%%%%%%%%%%%%%%%%%%%%%%%%%%%%%%%%%%%%%
%%%%%%%%%%%%%%%%%%%%%%%%%%%%%%%%%%%%%%%%%%%%%
\begin{definition}\label{stably degeneration}\cite[Definition 4.1]{Y10}
Let $\ulM, \ulN \in \SCMR$.
We say that $\ulM$ stably degenerates to $\ulN$ if there exists a maximal Cohen-Macaulay module ${\underline{Q}} \in \SCM (R\otimes _k V )$ such that $\L({\underline{Q}}) \cong {\underline{M \otimes _k K}}$ in $\SCM (R\otimes _k K)$ and $\R({\underline{Q}}) \cong \ulN$ in $\SCMR$.

We denote by $\ulM \sto \ulN$ if $\ulN$ is obtained from $\ulM$ by iterative stable degenerations, i.e.  there is a sequence of objects  $\ulL_0, \ulL_1, \ldots , \ulL_r$ in  $\SCMR$  such that  $\ulM \cong \ulL_0$, $\ulN \cong \ulL_r$  and  each $\ulL_{i}$  stably degenerates to  $\ulL_{i+1}$  for $0 \leq i < r$. 
\end{definition}

%%%%%%%%%%%%%%%%%%%%%%%%%%%%%%%%%%%%%%%%%%%%%
%%%%%%%%%%%%%%%%%%%%%%%%%%%%%%%%%%%%%%%%%%%%%
%%%%%%%%%%%%%%%%%%%%%%%%%%%%%%%%%%%%%%%%%%%%%
%%%% Remark of stably degenerations %%%%%%%%%
%%%%%%%%%%%%%%%%%%%%%%%%%%%%%%%%%%%%%%%%%%%%%
%%%%%%%%%%%%%%%%%%%%%%%%%%%%%%%%%%%%%%%%%%%%%
%%%%%%%%%%%%%%%%%%%%%%%%%%%%%%%%%%%%%%%%%%%%%
\begin{remark}\label{remark of stably degenerations}
Let $R$ be a Gorenstein local ring that is a $k$-algebra. 

\begin{itemize}
\item[(1)] Let $M, N \in \CMR$. 
If $M$ degenerates to $N$, then $\ulM$ stably degenerates to $\ulN$. 
Therefore that  $M \dego N$  forces that  $\ulM \sto \ulN$. 
(See \cite[Lemma 4.2]{Y10}.)

\item[(2)] Suppose that there is a triangle 
$$
\begin{CD}
 \ulL @>>> \ulM @>>> \ulN @>>> \ulL [1], 
\end{CD}
$$ 
in $\SCMR$. 
Then $\ulM$ stably degenerates to $\ulL \oplus \ulN$, thus $\ulM \sto \ulL \oplus \ulN$. 
(See \cite[Proposition 4.3]{Y10}.) 
\end{itemize}
\end{remark}

The following theorem proved by the second author \cite{Y10} shows the relation between stable degenerations and ordinary degenerations.

%%%%%%%%%%%%%%%%%%%%%%%%%%%%%%%%%%%%%%%%%%%%%
%%%%%%%%%%%%%%%%%%%%%%%%%%%%%%%%%%%%%%%%%%%%%
%%%%%%%%%%%%%%%%%%%%%%%%%%%%%%%%%%%%%%%%%%%%%
%%% Condition for stably degenerations %%%%%%
%%%%%%%%%%%%%%%%%%%%%%%%%%%%%%%%%%%%%%%%%%%%%
%%%%%%%%%%%%%%%%%%%%%%%%%%%%%%%%%%%%%%%%%%%%%
%%%%%%%%%%%%%%%%%%%%%%%%%%%%%%%%%%%%%%%%%%%%%
\begin{theorem}\label{conditions for stable degeneration}\cite[Theorem 5.1, 6.1, 7.1]{Y10}
Let $(R, \m, k)$ be a Gorenstein complete local $k$-algebra, where $k$ is an infinite field. 
Consider the following four conditions for maximal Cohen-Macaulay $R$-modules $M$ and $N$: 

\begin{itemize}
\item[(1)] $R^{m}\oplus M$ degenerates to $R^{n}\oplus N$ for some $m, n \in \N$.
\item[(2)] There is a triangle 
$$
\begin{CD}
\ulZ @>{\tiny 
\begin{pmatrix}
{\underline{\varphi}} \\
{\underline{\psi}}  \\
\end{pmatrix} 
}>> \ulM \oplus \ulZ \ @>>> \ulN @>>> \ulZ [1] \\ 
\end{CD}
$$ 
in $\SCMR$, where ${\underline{\psi}}$ is a nilpotent element of $\SEnd _R (Z)$.
\item[(3)] $\ulM$ stably degenerates to $\ulN$. 
\item[(4)] There exists an $X \in \CMR$ such that $M \oplus R^m \oplus X$ degenerates to $N \oplus R^n \oplus X$ for some $m, n \in \N$. 
\end{itemize}

Then, in general, the implications $(1) \Rightarrow (2)\Rightarrow (3)\Rightarrow (4)$ hold.
If $R$ is an isolated singularity, then $(2)$ and $(3)$ are equivalent. 
Furthermore, if $R$ is an artinian ring, then the conditions $(1), (2)$ and $(3)$ are equivalent. 
\end{theorem}

As one of the direct consequences of Theorem \ref{conditions for stable degeneration}, we have the following corollary.

%%%%%%%%%%%%%%%%%%%%%%%%%%%%%%%%%%%%%%%%%%%%%%%%%%%%
%%%%% category equivalence %%%%%%%%%%%%%%%%%%%%%%%%%
%%%%%%%%%%%%%%%%%%%%%%%%%%%%%%%%%%%%%%%%%%%%%%%%%%%%
\begin{corollary}\label{cor isolated singularity}\cite[Corollary 6.6]{Y10}
Let  $(R_1, \m_1, k)$ and  $(R_2, \m_2, k)$ be Gorenstein complete local $k$-algebras.
Assume that the both $R_1$ and  $R_2$  are isolated singularities, and that $k$  is an infinite field. 
Suppose there is a $k$-linear equivalence  $F : \SCM (R_1) \to \SCM (R_2)$  of triangulated categories. 
Then, for $\ulM,\  \ulN \in \SCM (R_1)$, 
$\ulM$  stably degenerates to $\ulN$ if and only if  
$F(\ulM)$  stably degenerates to $F(\ulN)$. 
\end{corollary}

We need several lemmas to prove Theorem \ref{type A}. 

%%%%%%%%%%%%%%%%%%%%%%%%%%%%%%%%%%%%%%%%%%%%%
%%%%%%%%%%%%%%%%%%%%%%%%%%%%%%%%%%%%%%%%%%%%%
%%%%%%%%%%%%%%%%%%%%%%%%%%%%%%%%%%%%%%%%%%%%%
%%%%% Lemma A %%%%%%%%%%%%%%%%%%%%%%%%%%%%%%%
%%%%%%%%%%%%%%%%%%%%%%%%%%%%%%%%%%%%%%%%%%%%%
%%%%%%%%%%%%%%%%%%%%%%%%%%%%%%%%%%%%%%%%%%%%%
\begin{lemma}\label{lemma A} 
Let $R$ be a Gorenstein complete local ring. 
If there is a triangle $\ulL \to \ulM \to \ulN \to \ulL[1]$ in $\SCMR$, then there exist non-negative integers $m$ and $n$ such that  $M \oplus R^m \exto L \oplus N \oplus R^n$. 
In this case, we have $[M \oplus R^m] = [L \oplus N \oplus R^n]$ in $K_0 (\modR)$. 
\end{lemma}

\begin{proof}
If there is a triangle $\ulL \to \ulM \to \ulN \to \ulL[1]$, then by definition there is a short exact sequence $0 \to L' \to M' \to N' \to 0$ where  $\ulL ' \cong \ulL$,  $\ulM ' \cong \ulM$ and $\ulN ' \cong \ulN$ in $\SCMR$. 
Thus  $M' \exto L' \oplus N'$. 
Since  $R$  is a complete local ring, we note that  $L'$ (resp. $M'$, $N'$)  is isomorphic to $L$ (resp. $M$, $N$) up to free summands, i.e. 
$L' \oplus R^{a'} \cong L \oplus R^{a}$, 
$M' \oplus R^{b'} \cong M \oplus R^{b}$, 
$N' \oplus R^{c'} \cong N \oplus R^{c}$ for integers $a$, $a'$, $b$, $b'$, $c$  and $c'$. 
Therefore we have 
$$
M \oplus R^{a'+b+c'} \cong 
M' \oplus R^{a'+b'+c'} \exto 
L' \oplus N' \oplus R^{a'+b'+c'} 
\cong L \oplus N \oplus R^{a+b'+c}. 
$$
\end{proof}

%%%%%%%%%%%%%%%%%%%%%%%%%%%%%%%%%%%%%%%%%%%%%
%%%%%%%%%%%%%%%%%%%%%%%%%%%%%%%%%%%%%%%%%%%%%
%%%%%%%%%%%%%%%%%%%%%%%%%%%%%%%%%%%%%%%%%%%%%
%%%%% Lemma B %%%%%%%%%%%%%%%%%%%%%%%%%%%%%%%
%%%%%%%%%%%%%%%%%%%%%%%%%%%%%%%%%%%%%%%%%%%%%
%%%%%%%%%%%%%%%%%%%%%%%%%%%%%%%%%%%%%%%%%%%%%
The next lemma will follow easily from the fact that the free $R$-module  $R$  is relatively injective in $\CMR$ if $R$ is Gorenstein. 
We leave its proof to the reader.

\begin{lemma}\label{lemma B}
Let $R$ be a Gorenstein local ring and let $0 \to L \to M \to N \to 0$ be an exact sequence in $\CMR$. 
Suppose that $L$ (resp. $N$) contains a free $R$-module $R^n$ as a direct summand. 
Then there is an exact sequence in $\CMR$  of the form 
$0 \to L' \to M' \to N \to 0$ (resp. $0 \to L \to M' \to N' \to 0$), 
where  $L' \oplus R^n  \cong  L$ (resp. $N' \oplus R^n \cong N$) and $M' \oplus R^n  \cong M$. 
 \end{lemma}

As a result of this lemma we obtain the following. 

\begin{corollary}\label{cor B}
Let $R$ be a Gorenstein local $k$-algebra and let $M, N \in \CMR$. 
If $M \exto N \oplus R^n$ for an integer $n$, then there is an $R$-module $M' \in \CMR$ such that $M \cong M' \oplus R^n$ and $M' \exto N$.     
\end{corollary}

We now consider the stable analogue of the degeneration by an extension. 

%%%%%%%%%%%%%%%%%%%%%%%%%%%%%%%%%%%%%%%%%%%%%
%%%%%%%%%%%%%%%%%%%%%%%%%%%%%%%%%%%%%%%%%%%%%
%%%%%%%%%%%%%%%%%%%%%%%%%%%%%%%%%%%%%%%%%%%%%
%%%%% Definition of triangle order %%%%%%%%%%
%%%%%%%%%%%%%%%%%%%%%%%%%%%%%%%%%%%%%%%%%%%%%
%%%%%%%%%%%%%%%%%%%%%%%%%%%%%%%%%%%%%%%%%%%%%
\begin{definition}\label{triangle order}
We say that $\ulM$ stably degenerates  by a triangle to $\ulN$, if there is a triangle of the form ${\underline{U}} \to \ulM \to {\underline{V}} \to {\underline{U}}[1]$ in $\SCMR$ such that ${\underline{U}}\oplus {\underline{V}} \cong \ulN$. 
We denote by $\ulM \trio \ulN$  if there is a finite sequence of modules 
${\underline{L_0}}, {\underline{L_1}}, \cdots , {\underline{L_r}}$  in $\SCMR$ such that 
 $\ulM \cong {\underline{L_0}}$, $\ulN \cong {\underline{L_{r}}}$  and  
each ${\underline{L_{i}}}$  stably degenerates by a triangle to ${\underline{L_{i+1}}}$ for $0 \leq i < r$. 
\end{definition}

It is obvious from Remark \ref{remark of stably degenerations}(2) that  $\ulM \trio \ulN$  implies  that $\ulM \sto \ulN$.

%%%%%%%%%%%%%%%%%%%%%%%%%%%%%%%%%%%%%%%%%%%%%
%%%%%%%%%%%%%%%%%%%%%%%%%%%%%%%%%%%%%%%%%%%%%
%%%%%%%%%%%%%%%%%%%%%%%%%%%%%%%%%%%%%%%%%%%%%
%%%%% Proposition of triangle order %%%%%%%%%
%%%%%%%%%%%%%%%%%%%%%%%%%%%%%%%%%%%%%%%%%%%%%
%%%%%%%%%%%%%%%%%%%%%%%%%%%%%%%%%%%%%%%%%%%%%
\begin{proposition}\label{prop of tri.}
Let $(R, \m, k)$ be a Gorenstein complete local ring and let $M , N \in \CMR$. 
Assume $[M] = [N]$ in $K_{0} (\modR)$. 
Then  $\ulM \trio \ulN$  if and only if $M \exto N$.  
\end{proposition}

\begin{proof}
The implication $M \exto N \Rightarrow \ulM \trio \ulN$ is clear, 
since, if there is an exact sequence $0 \to U \to L \to V \to 0$ in $\CMR$, then there is a triangle ${\underline{U}} \to {\underline{L}} \to {\underline{V}} \to {\underline{U}} [1]$ in $\SCMR$. 

To prove the other implication, 
assume  $\ulM \trio \ulN$.
Then there are ${\underline{L_0}}, {\underline{L_1}}, \cdots , {\underline{L_r}}$  in $\SCMR$ such that 
 $\ulM \cong {\underline{L_0}}$, $\ulN \cong {\underline{L_{r}}}$  and  
each ${\underline{L_{i}}}$  stably degenerates by a triangle to ${\underline{L_{i+1}}}$ for $0 \leq i < r$. 
It follows from Lemma \ref{lemma A} that 
$L_{i} \oplus R^{a_i} \exto L_{i+1} \oplus R^{b_i}$  for some integers $a_i$  and $b_i$. 
Thus we have 
$$
L_0 \oplus R^{a_0 + a_1 + \ldots + a_{r-1}} 
\exto 
L_{1}\oplus R^{b_0 + a_1 + \ldots + a_{r-1}} \exto \cdots \exto 
L_{r}\oplus R^{b_0 + b_1 + \ldots + b_{r-1}} 
$$
Since $M$ (resp. $N$) is isomorphic to $L_0$  (resp. $L_r$) up to free summands, we conclude from the above that   $M \oplus R^{m} \exto N \oplus R^n$  for some integers  $m$  and $n$. 

Since  there is a degeneration, we have 
 $[M ] + [R^{m}] = [N] + [R^n]$  in  $K_0(\modR )$, which forces  $[R^m] = [R^n]$  by the assumption.
Thus it follows  $m = n$. 
Therefore, by Corollary \ref{cor B} and by the Krull-Schmidt property, we have  $M \exto N$ as desired. 
\end{proof}

As a corollary of the proof of Proposition \ref{prop of tri.},  we have the following.

\begin{corollary}
Let $(R, \m, k)$ be a Gorenstein complete local ring. 
Then the relation  $\trio$  gives a well-defined partial order on the set of isomorphism classes of objects in  $\SCMR$. 
\end{corollary}

\begin{proof}
We have to show that $\ulM \trio \ulN$ and $\ulN \trio \ulM$ implies that $\ulM \cong \ulN$  for $\ulM, \ulN \in  \SCMR$. 
If  $\ulM \trio \ulN$, then it follows from the proof of Proposition \ref{prop of tri.} that  $M \oplus R^m \exto N \oplus R^n$  for some integers $m$, $n$. 
Likewise if  $\ulN \trio \ulM$ then $N \oplus R^{n'} \exto M \oplus R^{m'}$   for some integers $m'$, $n'$. 
Combining them, we have  $M \oplus R^{m+n'} \exto N \oplus R^{n+n'} \exto M \oplus R^{n+m'}$.   
Since there is a degeneration, all of these modules give the same class in $K_0(\modR)$, and as in the same argument in the proof of Proposition \ref{prop of tri.} we see that  $m+n' = n +m'$. 
Recall that  $\exto$  is a partial order on the set of isomorphism classes of objects in $\CMR$. 
(See also the comments after Definition \ref{ext order}.) 
Thus it is concluded that  $M \oplus R^{m+n'} \cong N \oplus R^{n+n'}$ in $\CMR$, and hence  $\ulM \cong \ulN$  in $\SCMR$.
\end{proof}

The following lemma is known as the Kn\"orrer's periodicity (cf. \cite[Theorem 12.10]{Y}).
%%%%%%%%%%%%%%%%%%%%%%%%%%%%%%%%%%%%%%%%%%%%%
%%%%%%%%%%%%%%%%%%%%%%%%%%%%%%%%%%%%%%%%%%%%%
%%%%%%%%%%%%%%%%%%%%%%%%%%%%%%%%%%%%%%%%%%%%%
%%% Kn\"orrer's periodicity %%%
%%%%%%%%%%%%%%%%%%%%%%%%%%%%%%%%%%%%%%%%%%%%%
%%%%%%%%%%%%%%%%%%%%%%%%%%%%%%%%%%%%%%%%%%%%%
\begin{lemma}\label{periodicity}
Let  $k$  be an algebraically closed field of characteristic $0$ and 
let $S = k[[x_0, x_1, \cdots , x_n]]$ be a formal power series ring. 
For a non-zero element  $f \in (x_0, x_1, \cdots , x_n)S$, we consider 
the two rings $R = S/(f)$  and $R^{\sharp} = S[[y, z]] /(f+y^2+z^2)$.  
Then the stable categories  $\SCM (R)$ and $\SCM (R^{\sharp})$ are equivalent as triangulated categories. 
\end{lemma}

%%%%%%%%%%%%%%%%%%%%%%%%%%%%%%%%%%%%%%%%%%%%%
%%%%%%%%%%%%%%%%%%%%%%%%%%%%%%%%%%%%%%%%%%%%%
%%%%%%%%%%%%%%%%%%%%%%%%%%%%%%%%%%%%%%%%%%%%%
%%%%% Proof of Thorem {type A} %%%%%%%%%%%%%%
%%%%%%%%%%%%%%%%%%%%%%%%%%%%%%%%%%%%%%%%%%%%%
%%%%%%%%%%%%%%%%%%%%%%%%%%%%%%%%%%%%%%%%%%%%%
\vspace{12pt}

Now we proceed to the proof of Theorem \ref{type A}.

Let  $k$ be an algebraically closed field of characteristic $0$ and 
let 
$$
R = k[[ x_0, x_1, x_2, \cdots , x_d ]] / (x_0 ^{n + 1} + x_1 ^2 +x_2 ^2 + \cdots + x_d ^2)
$$ 
as in the theorem, where we assume that $d$ is even. 
Suppose that $M \dego N$ for maximal Cohen-Macaulay $R$-modules $M$ and $N$. 
We want to show $M \exto N$.

Since $M \dego N$, we have $\ulM \sto \ulN$ in $\SCMR$  and  $[M] = [N]$  in $K_0 (\modR)$,  by Remarks \ref{remark of stably degenerations}(1) and  \ref{remark of degenerations}(1).  
Now let us denote $R' = k[[x_0]]/(x_0^{n+1})$, and we note that $\SCMR$ and $\SCM (R')$ are equivalent to each other as triangulated categories.  
In fact this equivalence is given by using $d/2$-times of Lemma \ref{periodicity}, since  $d$ is even.
Let $\Omega :\SCMR \to \SCM(R')$ be a triangle functor which gives the equivalence. 
Then, by virtue of Corollary \ref{cor isolated singularity}, we have  
$\Omega (\ulM)  \sto \Omega (\ulN)$ in $\SCM(R')$. 
Since  $R'$  is an artinian algebra, the equivalence  $(1) \Leftrightarrow (3)$  holds in Theorem \ref{conditions for stable degeneration}, and thus 
we have $\tilde{M} \oplus {R'}^m \dego \tilde{N} \oplus {R'}^n$, 
where  $\tilde{M}$  (resp. $\tilde{N}$) is a module in $\CM (R')$ with 
$\underline{\tilde{M}} \cong \Omega (\ulM)$ 
(resp. $\underline{\tilde{N}} \cong \Omega (\ulN)$) and $m$, $n$ are non-negative integers. 
It then follows from Corollary \ref{cor Jordan} that $\tilde{M} \oplus {R'}^m \exto \tilde{N} \oplus {R'}^n$. 
Hence, by Proposition \ref{prop of tri.},  we have that 
$\Omega (\ulM) \trio \Omega (\ulN)$ in $\SCM (R')$. 
Noting that the partial order  $\trio$  is preserved under a triangle functor, we see that  $\ulM \trio \ulN$ in $\SCMR$. 
Since $[M]=[N]$ in $K_0(\modR)$, applying Proposition \ref{prop of tri.},  we finally obtain that $M \exto N$. 
\qed

%%%%%%%%%%%%%%%%%%%%%%%%%%%%%%%%%%%%%%%%%%%%%
%%%%%%%%%%%%%%%%%%%%%%%%%%%%%%%%%%%%%%%%%%%%%
%%%%%%%%%%%%%%%%%%%%%%%%%%%%%%%%%%%%%%%%%%%%%
%%%%% Example of n=2 rank=3 %%%%%%%%%%%%%%%%%
%%%%%%%%%%%%%%%%%%%%%%%%%%%%%%%%%%%%%%%%%%%%%
%%%%%%%%%%%%%%%%%%%%%%%%%%%%%%%%%%%%%%%%%%%%%
\begin{example}
Let $R=k[[x_0, x_1, x_2]]/(x_{0}^{3} + x_{1}^{2} + x_{2}^{2})$, 
where  $k$  is an algebraically closed field of characteristic $0$. 
Let $\p$ and $\q$ be the ideals generated by $(x_{0}, x_{1}-\sqrt{-1}\ x_2)$ and $(x_{0}^{2}, x_{1}+\sqrt{-1}\ x_2)$ respectively. 
It is known that the set  $\{ R, \p, \q \}$  is a complete list of 
the isomorphism classes of indecomposable maximal Cohen-Macaulay modules over $R$. 
See \cite[Chapter 10]{Y}. 
We see from Theorem \ref{type A} that all degenerations in $\CMR$  are given by extensions. 
By this fact we can easily describe the degenerations in  $\CMR$. 
For example, the Hasse diagram of degenerations of maximal Cohen-Macaulay $R$-modules of rank $3$ is a disjoint union of the following diagrams:

$$
\xygraph
{
{\ R^3,} -[u]
{R\oplus \p \oplus \q}
(-[ul] {\p \oplus \p \oplus \p},-[ur] {\q \oplus \q \oplus \q})
}
\qquad 
\qquad
\xygraph
{
{\!\!\!R^2 \oplus \p ,}
(-[u] {R\oplus \q \oplus \q }
(
-[u] {\p \oplus \p  \oplus \q}
)
)
}
\qquad
\qquad 
\xygraph
{
{\!\!\!R^2 \oplus \q .}
(-[u] {R\oplus \p \oplus \p}
(
-[u] {\p \oplus \q \oplus \q}
)
)
}
$$
\end{example}

\begin{remark}
Theorem \ref{type A} is expected to hold without the assumption on $d$. 
Unfortunately, at the moment of writing the manuscript, the authors do not know  any appropriate proof for this.   
\end{remark}

\vspace{6pt}

%%%%%%%%%%%%%%%%%%%%%%%%%%%%%%%%%%%%%%%%%%%%%%%%%%%%%%%%%
%%%%%%%%%%%%%%%%%%%%%%%%%%%%%%%%%%%%%%%%%%%%%%%%%%%%%%%%%
%%%%%%%%%%%%%%%%%%%%%%%%%%%%%%%%%%%%%%%%%%%%%%%%%%%%%%%%%
%%%%%%%%%%%%%%%%%%%%%%% Section 3 %%%%%%%%%%%%%%%%%%%%%%%
%%%%%%%%%%%%%%%%%%%%%%%%%%%%%%%%%%%%%%%%%%%%%%%%%%%%%%%%%
%%%%%%%%%%%%%%%%%%%%%%%%%%%%%%%%%%%%%%%%%%%%%%%%%%%%%%%%%
%%%%%%%%%%%%%%%%%%%%%%%%%%%%%%%%%%%%%%%%%%%%%%%%%%%%%%%%%
%%%%%%%%%%%%%%%%%%%%%%%%%%%%%%%%%%%%%%%%%%%%%%%%%%%%%%%%%
\section{Extended orders}\label{Orders}
In the rest of this paper $R$ denotes a (commutative) Cohen-Macaulay complete local $k$-algebra, where  $k$  is any field. 

We shall show that any extended degenerations of maximal Cohen-Macaulay $R$-modules are generated by extended degenerations of Auslander-Reiten (abbr.AR) sequences if $R$  is of finite Cohen-Macaulay representation type. 
For the theory of AR sequences of maximal Cohen-Macaulay modules, we refer to \cite{Y}. 
First of all we recall the definitions of the extended orders generated respectively by  degenerations, extensions and AR sequences. 

%%%%%%%%%%%%%%%%%%%%%%%%%%%%%%%%%%%%%%%%%%
%%%%%%%%%%%%%%%%%%%%%%%%%%%%%%%%%%%%%%%%%%
%%%%%%%%%%%%%%%%%%%%%%%%%%%%%%%%%%%%%%%%%%
%%%%%%% Definition of Deg order %%%%%%%%%%
%%%%%%%%%%%%%%%%%%%%%%%%%%%%%%%%%%%%%%%%%%
%%%%%%%%%%%%%%%%%%%%%%%%%%%%%%%%%%%%%%%%%%
%%%%%%%%%%%%%%%%%%%%%%%%%%%%%%%%%%%%%%%%%%
\begin{definition}\label{Deg order}\cite[Definition 4.11, 4.13]{Y02}
The relation $\Dego $ on $\CMR$, which is called the extended degeneration order, is a partial order generated by the following rules: 
\begin{itemize}
\item[(1)] If $M \dego N$ then $M \Dego N$.
\item[(2)] $M \Dego N$ if and only if $M \oplus L \Dego N \oplus L$ for all 
 $L \in \CMR$. 
\item[(3)] $M \Dego N$ if and only if $M^n \Dego N^n$ for all natural numbers $n$.  
\end{itemize}
\end{definition}

%%%%%%%%%%%%%%%%%%%%%%%%%%%%%%%%%%%%%%%%%%
%%%%%%%%%%%%%%%%%%%%%%%%%%%%%%%%%%%%%%%%%%
%%%%%%%%%%%%%%%%%%%%%%%%%%%%%%%%%%%%%%%%%%
%%%% Definition of Ext and AR order %%%%%%
%%%%%%%%%%%%%%%%%%%%%%%%%%%%%%%%%%%%%%%%%%
%%%%%%%%%%%%%%%%%%%%%%%%%%%%%%%%%%%%%%%%%%
%%%%%%%%%%%%%%%%%%%%%%%%%%%%%%%%%%%%%%%%%%
\begin{definition}\label{Ext orders}\cite[Definition 3.6]{Y02}
The relation $\Exto $ on $\CMR$, which is called the extended extension order, is a partial order generated by the following rules: 
\begin{itemize}
\item[(1)] If $M \exto N$ then $M \Exto N$.
\item[(2)] $M \Exto N$ if and only if $M \oplus L \Exto N \oplus L$ for all $L \in \CMR$. 
\item[(3)] $M \Exto N$ if and only if $M^n \Exto N^n$ for all natural numbers $n$.  
\end{itemize}
\end{definition}

\begin{definition}\label{AR orders}\cite[Definition 5.1]{Y02}
The relation  $\ARo $  on $\CMR$, which is called the extended AR order, is a partial order generated by the following rules: 
\begin{itemize}
\item[(1)] If $0 \to X \to E \to Y \to 0$ is an AR sequence in $\CMR$, then $E \ARo X \oplus Y$.
\item[(2)] $M \ARo N$ if and only if $M \oplus L \ARo N \oplus L$ for all 
$L \in \CMR$. 
\item[(3)] $M \ARo N$ if and only if $M^n \ARo N^n$ for all natural numbers $n$.\end{itemize} 
\end{definition}

The following is the main theorem of this section.

%%%%%%%%%%%%%%%%%%%%%%%%%%%%%%%%%%%%%%%%%%
%%%%%%%%%%%%%%%%%%%%%%%%%%%%%%%%%%%%%%%%%%
%%%%%%%%%%%%%%%%%%%%%%%%%%%%%%%%%%%%%%%%%%
%%%%%%%% Main Theorem %%%%%%%%%%%%%%%%%%%%
%%%%%%%%%%%%%%%%%%%%%%%%%%%%%%%%%%%%%%%%%%
%%%%%%%%%%%%%%%%%%%%%%%%%%%%%%%%%%%%%%%%%%
%%%%%%%%%%%%%%%%%%%%%%%%%%%%%%%%%%%%%%%%%%
\begin{theorem}\label{theorem of AR order}
Let $R$ be a Cohen-Macaulay complete local $k$-algebra as above. 
Adding to this,  we assume that $R$ is of finite Cohen-Macaulay representation type, i.e. there are only a finite number of isomorphism classes of objects in $\CMR$. 
Then the following conditions are equivalent for  $M, N \in \CMR$:
\begin{itemize}
\item[(1)]
$M \Dego N$, 
\item[(2)]
$M \Exto N$, 
\item[(3)]
$M \ARo N$. 
\end{itemize} 
\end{theorem}

%%%%%%%%%%%%%%%%%%%%%%%%%%%%%%%%%%%%%%%%%%
%%%%%%%%%%%%%%%%%%%%%%%%%%%%%%%%%%%%%%%%%%
%%%%%%%%%%%% proof %%%%%%%%%%%%%%%%%%%%%%%
%%%%%%%%%%%%%%%%%%%%%%%%%%%%%%%%%%%%%%%%%%
%%%%%%%%%%%%%%%%%%%%%%%%%%%%%%%%%%%%%%%%%%
\vspace{6pt}
\noindent
{\it Proof.}  The implications $(3) \Rightarrow (2) \Rightarrow (1)$  are clear from the definitions. 

To prove $(1) \Rightarrow (2)$, 
it suffices to show that $M \Exto N$ whenever $M$ degenerates to $N$. 
If $M$ degenerates to $N$, then,  by virtue of Theorem \ref{Zwara sequence}, we have a short exact sequence
$
0 \to Z \to M \oplus Z \to N \to 0 
$
with  $Z \in \CMR$. 
Thus  $M \oplus Z \exto N \oplus Z$, hence $M \Exto N$. 

\vspace{6pt}
It remains to prove that  $(2) \Rightarrow (3)$, for which we need several preparations. 

\vspace{6pt}
%%%%%%%%%%%%%%%%%%%%%%%%%%%%%%%%%%%%%%%%%%
%%%%%%%% Auslander category %%%%%%%%%%%%%%
%%%%%%%%%%%%%%%%%%%%%%%%%%%%%%%%%%%%%%%%%%
Under the circumstances of Theorem \ref{theorem of AR order}, 
we consider the functor category  $\ModC$ and the Auslander category $\modC$ 
of $\CMR$. 
By definition, $\ModC$  is the category whose objects are contravariant additive functors from $\CMR$  to the category of Abelian groups and whose morphisms are natural transformations between functors. 
Note that  $\ModC$  is an Abelian category. 
The Auslander category  $\modC$  is a full subcategory of  $\ModC$  consisting of all finitely presented functors. 
Recall that a functor  $F \in \ModC$  is called finitely presented if 
there is an exact sequence in $\ModC$; 
$$
\Hom _R (\  , M) \to \Hom _R (\  , N) \to F \to 0, 
$$
for $M, N \in \CMR$. 

If there is a short exact sequence 
$0 \to L \to M \to N \to 0$ in $\CMR$, 
then the finitely presented functor $F$ is defined by the exact sequence in 
$\modC$; 
\begin{equation}\label{fp}
0 \to \Hom _R (\  , L) \to \Hom _R (\  , M) \to \Hom _R (\  , N) \to F \to 0. 
\end{equation}
Such a functor $F \in \modC$  satisfies $F(R) =0$. 
Conversely every element $F \in \modC$ with the property $F(R) =0$  is obtained in this way from a short exact sequence in $\CMR$.  

If  $0 \to X \to E \to Y \to 0$  is an AR sequence in $\CMR$, then the functor $S$  defined by an exact sequence 
\begin{equation}\label{simple}
0 \to \Hom _R (\  , X) \to \Hom _R (\  , E) \to \Hom _R (\  , Y) \to S \to 0 
\end{equation}
is a simple object in $\modC$  and all the simple objects in $\modC$ are obtained in this way from AR sequences.

It is proved in \cite[(13.7.4)]{Y} that 

\vspace{3pt} 
\noindent 
$(4.3) \ $ every object $F$  in  $\modC$  with  $F(R) =0$  has a composition series, i.e. there is a filtration by subobjects  $0 \subset F_1 \subset F_2 \subset \cdots \subset F_n = F$ such that each $F_{i}/F_{i-1}$ is a simple object in $\modC$. 
\vspace{3pt}

Now we consider a free Abelian group 
$$
\GR = \bigoplus \ \Z \cdot X,
$$
where $X$ runs through all isomorphism classes of indecomposable objects in $\CMR$. 
There is a group homomorphism 
$$
\gamma : \GR \to K_{0}(\modC ),  
$$
defined by $\gamma (M) =[ \Hom _R ( \ ,M)]$  for $M \in \CMR$. 

It is prove in \cite[Theorem 13.7]{Y} that 

\vspace{3pt} 
\noindent 
$(4.4) \ $ the group homomorphism $\gamma$  is injective. 
\vspace{3pt}

\noindent 
(See the first paragraph of the proof of Theorem 13.7 in \cite{Y}. )

The following lemma is essentially due to Auslander and Reiten \cite{AR86}. 
%%%%%%%%%%%%%%%%%%%%%%%%%%%%%%%%%%%%%%%%%%
%%%%%%%%%%%%%%%%%%%%%%%%%%%%%%%%%%%%%%%%%%
%%%%%%%%%%%%%%%%%%%%%%%%%%%%%%%%%%%%%%%%%%
%%%%%%%% Lemma of EX = AR %%%%%%%%%%%%%%%%
%%%%%%%%%%%%%%%%%%%%%%%%%%%%%%%%%%%%%%%%%%
%%%%%%%%%%%%%%%%%%%%%%%%%%%%%%%%%%%%%%%%%%
%%%%%%%%%%%%%%%%%%%%%%%%%%%%%%%%%%%%%%%%%%
\begin{lemma}\label{EX = AR}
Under the same assumptions on $R$ as in Theorem \ref{theorem of AR order}, 
let  $0 \to L \to M \to N \to 0$  be a short exact sequence in $\CMR$. 
Then there are a finite number of AR sequences in $\CMR$; 
$$
0 \to X_i \to E_i  \to Y_i  \to 0 \quad (1 \leq i \leq n), 
$$ 
such that there is an equality in $\GR$; 
$$
L - M + N = \sum  _{i = 1}^{n} (X_i - E_i + Y_i). 
$$
\end{lemma}

\begin{proof}
Consider the functor $F \in \modC$ defined by (\ref{fp}). 
By virtue of (4.3), there is a filtration by subobjects  $0 \subset F_1 \subset F_2 \subset \cdots \subset F_n = F$ such that each $F_{i}/F_{i-1}$ is a simple object in $\modC$. 
Now take an AR sequence  $0 \to X_i \to E_i \to Y_i \to 0$ corresponding to the simple functor $F_i/F_{i-1}$ for each $i$. 
Thus we have equalities in $K_{0} (\modC )$: 
$$
\begin{array}{ll}
[\Hom_R(&  , L)] - [\Hom_R(\  , M)] + [\Hom_R(\  , N)] = \sum ^{n}_{i = 1} [F_{i}/F_{i-1}] \vspace{6pt}\\
&= \sum ^{n}_{i = 1} (\Hom_R[(\  , X_i)]  - [\Hom_R(\  , E_i)]+ [\Hom_R(\  , Y_i)])
\end{array}
$$ 
This is equivalent to that 
$$
\gamma ( L-M+N)= \sum ^{n}_{i = 1} \gamma  ( X_i- E_i+Y_i). 
$$
Since $\gamma$ is injective (4.4), we have 
the equality $L -M + N = \sum ^{n}_{i = 1}(X_i  - E_i + Y_i)$ in $\GR$. 
\end{proof}
%%%%%%%%%%%%%%%%%%%%%%%%%%%%%%%%%%%%%%%%%%
%%%%%%%%%%%%%%%%%%%%%%%%%%%%%%%%%%%%%%%%%%
%%%%%%%%%%%%%%% proof %%%%%%%%%%%%%%%%%%%%
%%%%%%%%%%%%%%%%%%%%%%%%%%%%%%%%%%%%%%%%%%
%%%%%%%%%%%%%%%%%%%%%%%%%%%%%%%%%%%%%%%%%%
Now we shall continue the proof of Theorem \ref{theorem of AR order}.  
It remains to show that $M \Exto N$ implies that $M \ARo N$. 
We have only to prove that if $M$ degenerates by an extension to $N$, then $M \ARo N$. 
Assuming that $M$ degenerates by an extension to $N$, we have a short exact sequence 
$0 \to N_1 \to M \to N_2 \to 0$  in $\CMR$  with  $N \cong N_1 \oplus N_2$. 
Then, by Lemma \ref{EX = AR}, there are a finite number of AR sequences in $\CMR$; 
$$
0 \to X_i \to E_i  \to Y_i  \to 0 \quad (1 \leq i \leq n), 
$$ 
such that there is an equality in $\GR$; 
$$
N_1 - M + N _2 =  \sum  _{i = 1}^{n}  (X_i - E_i + Y_i). 
$$
This equality is equivalent to that there is an isomorphism of $R$-modules;
$$
M \oplus \sum  _{i = 1}^{n}(X_{i} \oplus Y_{i}) \cong  N_1 \oplus N_2 \oplus \sum  _{i = 1}^{n}E_{i}. 
$$
Since $E_{i}  \ARo (X_{i} \oplus Y_{i})$ for all $1 \leq i \leq n$, we have 
$$
\begin{array}{ll}
M\oplus \sum  _{i = 1}^{n}(X_{i} \oplus Y_{i}) &\cong  N_1 \oplus N_2 \oplus \sum  _{i = 1}^{n}E_{i}  \vspace{3pt} \\
&\ARo N_1 \oplus N_2 \oplus \sum  _{i = 1}^{n}(X_{i} \oplus Y_{i}). \\
\end{array}
$$
Therefore  $M \ARo N_1 \oplus N_2 \cong N$, and the proof is completed.  
\qed

%%%%%%%%%%%%%%%%%%%%%%%%%%%%%%%%%%%%%%%%%%
%%%%%%%%%%%%%%%%%%%%%%%%%%%%%%%%%%%%%%%%%%
%%%%%%%%%%%%%%%%%%%%%%%%%%%%%%%%%%%%%%%%%%
%%%% Remark of results of Yoshino %%%%%%%%
%%%%%%%%%%%%%%%%%%%%%%%%%%%%%%%%%%%%%%%%%%
%%%%%%%%%%%%%%%%%%%%%%%%%%%%%%%%%%%%%%%%%%
%%%%%%%%%%%%%%%%%%%%%%%%%%%%%%%%%%%%%%%%%%
\begin{remark}
In the paper \cite{Y02}, the second author introduced the order relation $\homo$ as well. 
Adding to the assumption that $R$  is of finite Cohen-Macaulay representation type, 
if we assume further conditions on $R$, such as $R$ is an integral domain of dimension $1$  or  $R$ is of dimension $2$, 
 then he showed that $\homo$ is also equal to any of  $\ARo$, $\Exto$ and $\Dego$. 
\end{remark}

%%%%%%%%%%%%%%%%%%%%%%%%%%%%%%%%%%%%%%%%%%%%%%%%%%
%%%%%%%%%%%%%%%%%%%%%%%%%%%%%%%%%%%%%%%%%%%%%%%%%%
%%%%%%%%%%%%%%%%%%%%%%%%%%%%%%%%%%%%%%%%%%%%%%%%%%
%%%%%%%%%% References %%%%%%%%%%%%%%%%%%%%%%%%%%%%
%%%%%%%%%%%%%%%%%%%%%%%%%%%%%%%%%%%%%%%%%%%%%%%%%%
%%%%%%%%%%%%%%%%%%%%%%%%%%%%%%%%%%%%%%%%%%%%%%%%%%
%%%%%%%%%%%%%%%%%%%%%%%%%%%%%%%%%%%%%%%%%%%%%%%%%%


\begin{thebibliography}{99}

\bibitem{AR86}
{\sc M. ~Auslander} and {\sc I. ~Reiten}, 
{\it Grothendieck groups of algebras and orders}. 
J. Pure Appl. Algebra {\bf 39} (1986), 1--51.


\bibitem{B96}
{\sc K. ~Bongartz}, 
{\it On degenerations and extensions of finite-dimensional modules}.
Adv. Math. {\bf 121} (1996), 245--287.


\bibitem{H}
{\sc D.~Happel}, 
{\it Triangulated categories in the representation theory of finite-dimensional algebras}, 
London Mathematical Society Lecture Note Series {\bf 119}. 
Cambridge University Press, Cambridge, 1988. x+208 pp. 

\bibitem{Mac}
{\sc I.G.Macdonald}, 
{\it Symmetric functions and Hall polynomials, Second edition. With contributions by A. Zelevinsky}, 
Oxford Mathematical Monographs. Oxford Science Publications. 
The Clarendon Press, Oxford University Press, New York, 1995. x+475 pp.

\bibitem{R86}
{\sc C. ~Riedtmann}, 
{\it Degenerations for representations of quivers with relations}. 
Ann. Scient. $\acute{E}$cole Norm. Sup. $4^e$ s$\grave{e}$rie {\bf 19} (1986), 275--301.

\bibitem{Y}
{\sc Y.~Yoshino}, 
{\it Cohen-Macaulay Modules over Cohen-Macaulay Rings}, 
London Mathematical Society Lecture Note Series {\bf 146}. 
Cambridge University Press, Cambridge, 1990. viii+177 pp. 


\bibitem{Y02}
{\sc Y. ~Yoshino}, 
{\it On degenerations of Cohen-Macaulay modules}. 
J. Algebra {\bf 248} (2002), 272--290.


\bibitem{Y04}
{\sc Y. ~Yoshino}, 
{\it On degenerations of modules}. 
J. Algebra {\bf 278} (2004), 217--226.


\bibitem{Y06}
{\sc Y. ~Yoshino}, 
{\it Degeneration and G-dimension of modules}. 
Lecture Notes Pure Applied Mathematics vol. 244, `Commutative algebra' Chapman and Hall/CRC (2006), 259--265. 


\bibitem{Y10}
{\sc Y.~Yoshino}, 
{\it Stable degenerations of Cohen-Macaulay modules}, 
to appear in Journal of Algebra (2011), arXiv1012.4531. 


\bibitem{Z98}
{\sc G.~Zwara},
{\it A degeneration-like order for modules}. 
Arch. Math. {\bf 71} (1998), 437--444.


\bibitem{Z99}
{\sc G. ~Zwara}, 
{\it Degenerations for modules over representation-finite algebras}. 
Proc. Amer. Math. Soc. {\bf 127} (1999), 1313--1322.


\bibitem{Z00}
{\sc G.~Zwara},
{\it Degenerations of finite-dimensional modules are given by extensions}. 
Compositio Math. {\bf 121} (2000), 205--218.

\end{thebibliography}
\end{document}